\documentclass[11pt]{amsart}
\usepackage{xcolor}

\definecolor{grn}{rgb}{0,0.6,0}
\definecolor{mrn}{rgb}{0.3,0,0}
\definecolor{blue}{rgb}{0,0,0.7}
\definecolor{Mygray}{rgb}{0.75,0.75,0.75}
\definecolor{auburn}{rgb}{0.43, 0.21, 0.1}
\definecolor{britishracinggreen}{rgb}{0.0, 0.26, 0.15}
\definecolor{taupe}{rgb}{0.28, 0.24, 0.2}
\usepackage{amsmath,amssymb}
\newtheorem{theorem}{Theorem}[section]
\newtheorem{propn}[theorem]{Proposition}
\newtheorem{cor}[theorem]{Corollary}
\newtheorem{lemma}[theorem]{Lemma}

\newtheorem{rmk}[theorem]{Remark}

\newcommand{\eps}{\varepsilon}
\newcommand{\Z}{\mathbb{Z}}
\newcommand{\Q}{\mathbb{Q}}

\newcommand{\Mod}[1]{\ (\mathrm{mod}\ #1)}
\newcommand{\p}{\mathfrak{p}}
\newcommand{\q}{\mathfrak{q}}

\usepackage{latexsym,amssymb}
\usepackage{amssymb}
\usepackage{graphicx}
\usepackage{ textcomp }
\usepackage[left=2cm,right=2cm,top=2cm,bottom=2cm]{geometry}
\usepackage{stmaryrd}
\usepackage{tikz}
\usepackage{tikzpagenodes} 
\usepackage{lipsum}
\usepackage{multicol}

\begin{document}
\baselineskip=14.5pt
\title[ $\Z_2$-extensions with $\Z/2\Z$ as $2$-class group]{ $\Z_2$-extension of real quadratic fields with $\Z/2\Z$ as $2$-class group at each layer}

\author{H Laxmi and Anupam Saikia}
\address[H Laxmi and Anupam Saikia]{Department of Mathematics, Indian Institute of Technology Guwahati, Guwahati - 781039, Assam, India}

\email[H Laxmi]{hlaxmi@iitg.ac.in}

\email[Anupam Saikia]{a.saikia@iitg.ac.in}
\renewcommand{\thefootnote}{}



\footnote{\emph{We confirm that all the data are included in the article.}}

\renewcommand{\thefootnote}{\arabic{footnote}}
\setcounter{footnote}{0}

\begin{abstract}
  
Let $K= \mathbb{Q}(\sqrt{d})$ be a real quadratic field with $d$ having three distinct prime factors. We show that the $2$-class group of each layer in the $\Z_2$-extension of $K$ is $\Z/2\Z$ under certain elementary assumptions on the prime factors of $d$. In particular, it validates Greenberg's conjecture on the vanishing of the Iwasawa $\lambda$-invariant for a new family of infinitely many real quadratic fields. 

\end{abstract}

\maketitle
\section{Introduction}
The investigation of class groups of number fields is one of the frequently explored problems in algebraic number theory. Iwasawa examined the variation of the $p$-Sylow subgroups of class groups of the intermediate fields in a $\Z_p$-extension of a number field for a prime $p$. For a number field $K$, a Galois extension $K_\infty$ is called a $\Z_p$-extension if the Galois group  $\mbox{Gal}(K_\infty/K)$ is topologically isomorphic to $\Z_p$, the additive group of $p$-adic integers. For each natural number $n$, $K_{\infty}$ contains a unique subfield $K_n$ such that $\mbox{Gal}(K_n/K)$ is isomorphic to $\Z/p^n\Z$. Let $A(K_n)$ denote the $p$-Sylow subgroup of the class group of $K_n$ and $p^{e_n}$ denote its order. Iwasawa \cite{iwasawa} showed that for sufficiently large positive integers $n$, there exist non-negative constants $\mu(K_{\infty}/K)$, $\lambda(K_{\infty}/K)$, and $\nu(K_{\infty}/K)$ such that 
$$e_n = \mu(K_{\infty}/K)\cdot p^n + \lambda(K_{\infty}/K)\cdot n + \nu(K_{\infty}/K).$$   
Iwasawa obtained his result by studying the structure of the inverse limit $X_\infty = \lim\limits_{\substack{ \longleftarrow \\ n}}A(K_n)$, formed with respect to the norm maps, as a module over a ring $\Lambda$, known as the Iwasawa algebra for the extension $K_\infty/K$. The Iwasawa algebra is defined as $\Lambda = \lim\limits_{\substack{ \longleftarrow \\ n}}\Z_p[\Gamma_n]$, where $\Gamma_n$ denotes ${\rm{Gal}}(K_n/K)$ and the inverse limit is taken with respect to the natural projection maps from $\Gamma_m$ to $\Gamma _n$ for $m\geq n$. The natural action of $\Z_p[\Gamma_n]$ on $A(K_n)$ is compatible as we vary $n$, hence $X_\infty$ has the structure of a $\Lambda$-module. It turns out that $X_\infty$ is a finitely generated torsion $\Lambda$-module, and the constants $\mu(K_{\infty}/K)$ and $\lambda(K_{\infty}/K)$ are the Iwasawa invariants associated with the $\Lambda$-module $X_\infty$ \cite[Theorem 13.12]{washington_book}. 

\smallskip

For a totally real field $K$, there is a unique $\Z_p$-extension known as the cyclotomic $\Z_p$-extension of $K$. For a prime $p$, let $\zeta_{p^{n+1}}$ denote a primitive $p^{n+1}$-th root of unity in $\mathbb{C}$. For $K=\Q$, let $\Q_0 := \Q$ and $\Q_n$ be the unique extension of degree $p^n$ over $\Q$ contained in $\Q(\zeta_{p^{n+1}})$ when $p$ is an odd prime, and $\Q_n$ be the maximal real subfield of $\Q(\zeta_{2^{n+2}})$ for $p=2$. The extension $\Q_{\infty}:=\bigcup\limits_{n = o}^{\infty}\Q_n$ is known as the cyclotomic $\Z_p$-extension of $\Q$. For an arbitrary number field $K$, its cyclotomic $\Z_p$-extension $K_{\infty}$ is the infinite union $\bigcup\limits_{n = o}^{\infty}K_n$, where $K_n$ is the compositum of $K$ with $\Q_n$ for each $n$. Iwasawa conjectured that $\mu(K_\infty/K)$ must vanish for the cyclotomic $\Z_p$-extension $K_\infty$ over any number field $K$. In \cite{greenberg}, Greenberg conjectured that both $\mu(K_\infty/K)$ and $\lambda(K_\infty/K)$ must vanish for the $\Z_p$-extension $K_\infty/K$ when $K$ is a totally real field.   Following these conjectures, Ferrero and Washington proved that $\mu(K_\infty/K) = 0$ for the cyclotomic $\Z_p$-extension of a number field $K$ when $K/\Q$ is an abelian extension (cf. \cite{ferrero-washington}). Greenberg's conjecture has not been completely settled for number fields other than $\Q$. Some of the partial progress towards Greenberg's conjecture can be found in \cite{asjc-rama}, \cite{fukuda-komatsu}, \cite{gras2}, \cite{ichimura2}, \cite{kraft-schoof}, \cite{kumakawa2}, \cite{mizu_paper}, \cite{mouhib}, \cite{mouhib-mova}, \cite{nishino}, \cite{ozaki-taya}, \cite{yamamoto}  etc.

\smallskip
Henceforth, $K$ will denote a real quadratic field $\Q(\sqrt{d})$ and  $K_\infty$ will denote the (cyclotomic) $\Z_2$-extension of $K$. We shall simply write $\mu$ and $\lambda$ to denote the corresponding Iwasawa invariants for the $\Z_2$-extension $K_\infty/K$. The intermediate fields $K_n$ are of the form $\Q(\sqrt{d}, a_n)$, where $a_0 = 0$, $a_n = \sqrt{ 2 + a_{n-1}}$. Thus, $K_0 = K$, $K_1 = \Q(\sqrt{2}, \sqrt{d})$, $K_2 = \Q(\sqrt{2 + \sqrt{2}}, \sqrt{d})$, and so on. 
The field $K_1$ is a bi-quadratic extension of $\Q$, hence its arithmetic can be studied by examining that of its subfields $K$, $\Q(\sqrt{2})$, and $\Q(\sqrt{2d})$. For more about $2$-class groups and multi-quadratic fields, interested readers may refer to \cite{azizi}, \cite{ben-lem-sny}, \cite{brown-parry} and \cite{ouyang-zhang}.

\smallskip

In \cite{mouhib-mova}, Mouhib and Movahhedi considered the $\Z_2$-extension of $K = \Q(\sqrt{\ell_1\ell_2\ell_3})$, where $\ell_1$, $\ell_2$ and $\ell_3$ are distinct primes satisfying  $\ell_1 \equiv 5 \pmod 8$, $\ell_2 \equiv 3 \pmod 8$, and $\ell_3 \equiv 3 \pmod 4$. They proved that the Iwasawa module $X_{\infty}$ corresponding to the $\Z_2$-extension of $K$ is a (finite or infinite) cyclic group (cf. Theorem 3.8, part (iv), \cite{mouhib-mova}). Further, they proved that if $\ell_3 \equiv 7 \pmod 8$, the corresponding $\lambda$-invariant is equal to $0$ (cf. Theorem 4.4, \cite{mouhib-mova}), and thus, $X_\infty$ is not only cyclic but finite as well in this case. Driven by their results, we focus on the finer structure of $X_{\infty}$ for the aforementioned fields. In particular, we show that $X_{\infty}$ is finite and cyclic of order $2$  when the primes satisfy certain Legendre symbol conditions. We also verify Greenberg's conjecture on vanishing of $\lambda$-invariant for some additional cases. In this article, the number fields $K$ that we shall revolve around are of the following kind:
\begin{equation}\label{cond1}
K = \Q(\sqrt{p_1q_1q_2}),\ p_1 \equiv 5 \Mod{8},\ q_1 \equiv 3 \Mod{8},\ q_2 \equiv 3 \Mod{8},
\end{equation}
\begin{equation}\label{cond2}
K = \Q(\sqrt{p_1q_1q_2}),\ p_1 \equiv 5 \Mod{8},\ q_1 \equiv 7 \Mod{8},\ q_2 \equiv 3 \Mod{8},
\end{equation}
where $p_1$, $q_1$ and $q_2$ denote three distinct primes.
We shall use $A_n$ to denote the $2$-class group of $K_n$, where $K_n$ is the $n^{\rm{th}}$ layer in the $\Z_2$-extension of $K$, in the context of the above conditions. In general, for any number field $L$, we shall use the symbol $\mathcal{C}l(L)$ to denote its class group, $A(L)$ for its $2$-class group, and $h(L)$ for its class number. We prove the following results in this article:

\smallskip

\begin{theorem}\label{(5,3,3,-1)}
Let $K = \Q(\sqrt{p_1q_1q_2})$ be a real quadratic number field such that $p_1 \equiv 5 \pmod 8$, $q_1, q_2 \equiv 3 \pmod 8$. Then, $\#A_1 = \#A_0$ if $\left(\dfrac{q_1q_2}{p_1}\right) = -1$. 
\end{theorem}
 
\begin{cor}\label{lambda of 5,3,3,-1}
Let $K = \Q(\sqrt{p_1q_1q_2})$ and $F = \Q(\sqrt{2p_1q_1q_2})$ be real quadratic number fields such that $p_1 \equiv 5 \pmod 8$, $q_1, q_2 \equiv 3 \pmod 8$, and $\left(\dfrac{q_1q_2}{p_1}\right) = -1$. Then, $A_n \cong \Z/2\Z$ for all $n\geq 1$, and the Iwasawa module $X_{\infty}$ corresponding to the $\Z_2$-extension of $K$ is isomorphic to $\Z/2\Z$. In particular, the $\lambda$-invariant for the $\Z_2$-extension of $K$ as well as $F$ is equal to $0$. 
\end{cor}
\noindent Note that while cyclicity of $X_\infty$ under the assumptions of Corollary \ref{lambda of 5,3,3,-1} was shown in \cite{mouhib-mova}, our result proves that $X_\infty$ is in fact a finite group of order $2$, resulting in vanishing of the $\lambda$-invariant. 

\begin{theorem}\label{(5,3,3,1)}
Let $K = \Q(\sqrt{p_1q_1q_2})$ be a real quadratic number field such that $p_1 \equiv 5 \pmod 8$, $q_1, q_2 \equiv 3 \pmod 8$, $\left( \dfrac{q_1}{p_1} \right) =1$, and $\left( \dfrac{q_2}{p_1} \right) =1$. Then, the ideal $\mathfrak{p}_1$ in $K$ lying above $p_1$ is principal if and only if $\#A_1 \neq \#A_0$. 
\end{theorem}

\begin{cor}\label{lambda of (5,3,3,1)}
Let $K = \Q(\sqrt{p_1q_1q_2})$ be a real quadratic number field such that $p_1 \equiv 5 \pmod 8$, $q_1, q_2 \equiv 3 \pmod 8$, $\left( \dfrac{q_1}{p_1} \right) =1$, and $\left( \dfrac{q_2}{p_1} \right) =1$. If the ideal $\mathfrak{p}_1$ is non-principal, then the Iwasawa module $X_{\infty}$ corresponding to $K$ is isomorphic to $\Z/2^m\Z$ for some $m \geq 2$. Consequently, the Iwasawa $\lambda$-invariant for such fields is equal to 0 if there are no integers $a$ and $b$ such that $a^2 - b^2p_1q_1q_2 = 4p_1$. Under these circumstances, the Iwasawa module corresponding to $F = \Q(\sqrt{2p_1q_1q_2})$ has the same structure, with $\lambda = 0$. 
\end{cor}

\begin{theorem}\label{(5,7,3)}
Let $K = \Q(\sqrt{p_1q_1q_2})$ be a real quadratic number field such that $p_1 \equiv 5 \pmod 8$, $q_1 \equiv 7 \pmod 8$, $q_2 \equiv 3 \pmod 8$, and $\left(\dfrac{q_1}{p_1}\right) = -1$.  Then, $A_n \cong \Z/2\Z$ for all $n\geq 1$, and the Iwasawa module $X_{\infty}$ corresponding to the $\Z_2$-extension of $K$ is isomorphic to $\Z/2\Z$. 
\end{theorem}
\noindent It also follows from Theorem \ref{(5,7,3)} that the $\lambda$-invariant associated with the $Z_2$-extension of $K$ as well as $F=\Q(\sqrt{2p_1 q_1 q_2})$ vanishes, as proven earlier in \cite{mouhib-mova}.  

\section{Preliminaries}
For a number field $K$, the $2$-part of the class group of $K$ can be effectively studied using genus theory. 
For a quadratic extension $K/k$ of number fields, the genus formula connects the order of a subgroup of the $2$-class group $A(K)$ with the order of $A(k)$. Other entities that appear in this formula are the number of places ramified in $K/k$ and the index of the norm of units of $K$ in the group of units in $k$. 
We first state the formula. 
\begin{theorem}(Genus Formula)(\cite{chevalley}, \cite[Theorem 2.5]{mizu_thesis})\label{genusfor}
Let $K/k$ be a quadratic extension of number fields with Galois group $G = {\rm{Gal}} \left(K/k\right)$. Let $A(K)^{G}$ be the subgroup of $A(K)$ consisting of the ideal classes that are fixed by the action of $G$ on $A(K)$. Let $N_{K/k}$ denote the norm map from $K$ to $k$. Let $E(K)$ and $E(k)$ be the unit groups of $K$ and $k$, respectively. If $t$ is the number of places of $k$ ramified in $K$, then 
\begin{align*}
\#A(K)^{G} &= \#A(k) \times \dfrac{2^{t-1}}{\left[ E(k):E(k)\cap N_{K/k}K^{\times} \right]}.
\end{align*}
\end{theorem}

Let $p$ be any prime number and $A(F)$ be the $p$-sylow subgroup of the class group of a number field $F$. By $\text{rank}_p A(F)$ or $p$-rank of $A(F)$, we mean the dimension of $A(F)/pA(F)$ as a vector space over the field $\Z/p\Z$. The following proposition enables us to examine the $2$-rank of a quadratic extension of $\Q$, and we outline its proof though it is well-known in the literature.
\begin{propn}\label{rmk to genus}
If $K/k$ is a quadratic extension of number fields and the image of the lifting map $j: A(k) \rightarrow A(K)$ is trivial, then the non-trivial element of $G={\rm{Gal}}(K/k)$ acts as $-1$ on $A(K)$. In that case, $A(K)^{G}$ is the subgroup of elements of order $2$. Consequently, $$\# A(K)^{G} = \# \left( A(K)/2A(K)\right) = 2^{\ {\rm{rank}_2}A(K)}.$$
\end{propn}
\begin{proof}
    Let $\sigma$ be the generator of ${\rm{Gal}}(K/k)$, and $[\mathfrak{P}]$ be an ideal class in $A(K)$ for a prime ideal $\mathfrak{P}$ of $K$. Let $\p$ be prime in $k$ lying below $\mathfrak{P}$. Let $f$ be the residue degree of $\mathfrak{P}$ in $K/k$. Since the lifting map is trivial, we have 
    $$\sigma[\mathfrak{P}]\cdot[\mathfrak{P}] = j([\p^f])=id.$$ Hence, $\sigma$ acts as $-1$ on $A(K)$. Therefore, $A(K)^G = A(K)[2]$ and the result follows.
\end{proof}

The class number of $\Q$ is equal to $1$, and its $2$-class group is trivial. For a quadratic extension $K=\Q(\sqrt{d})$, Proposition \ref{rmk to genus} implies that  
\begin{equation*}
2^{\ {\rm{rank}_2}A(K)} = \dfrac{ 2^{t-1}}{\left[E(\Q):E(\Q)\cap N_{K/\Q}(K^{\times})\right]},
\end{equation*}
where $t$ is the number of rational primes ramified in $K/\Q$. The group $A(K)$ is cyclic if and only if its $2$-rank is equal to $1$. Since $E(\Q) = \{-1, 1\}$, the index in the denominator of the formula is either $1$ or $2$. In such a situation, we have $2^{t-1} = 2$ or $4$ and $t=2$ or $3$. Here, we emphasise that our cases of interest require $t=3$. 

\begin{rmk}\label{A_n is cyclic}
From a result proved by Mouhib and Movahhedi (cf. \cite[Theorem 3.8, part $iv$]{mouhib-mova}), we infer that $A(K_n)$ is cyclic for all $n \geq 0$, where $K$ satisfies condition (\ref{cond1}) or (\ref{cond2}) of Section 1.
\end{rmk}

\smallskip

In order to find the structure of the Iwasawa module $X_\infty$, we appeal to a result of Fukuda on the stability of rank and order of $p$-class groups in a $\Z_p$-extension of any number field $L$, where $p$ is any prime number.

\begin{theorem}\cite[Theorem 1]{fukuda}\label{fukuda's result}
Let $p$ be a prime number. Let $L$ be a number field and let $L_{\infty}/L$ be a $\Z_p$-extension of $L$.  Let $A(L_n)$ denote the $p$-Sylow subgroup of the $n$-th layer $L_n$ in the extension $L_\infty/L$. Let $n_0 \geq 0$ be an integer such that any prime of $L_{\infty}$ which is ramified in $L_{\infty}/L$ is totally ramified in $L_{\infty}/L_{n_0}$. Then the following hold. 
\begin{enumerate}
\item If there exists an integer $n \geq n_0$ such that $\#A(L_{n+1}) = \#A(L_n)$, then $\#A(L_m) = \#A(L_n)$ for all $m \geq n$. In particular, both the Iwasawa invariants $\mu(L_\infty/L)$ and $\lambda(L_\infty/L)$ vanish. 
  
  \smallskip
  
\item If there exists an integer $n \geq n_0$ such that ${\rm{rank}}_{p}A(L_{n+1}) = {\rm{rank}}_{p}A(L_{n})$, then ${\rm{rank}}_{p}A(L_{m}) = {\rm{rank}}_{p}A(L_{n})$ for all $m \geq n$. In particular, the Iwasawa invariant $\mu(L_\infty/L)$ vanishes.
\end{enumerate}
\end{theorem}


The group of units in the ring of integers of extension of a number field is a crucial ingredient in the study of class groups. In particular, we are going to use Kuroda-Kubota's class number formula stated below. 
\begin{theorem}(\cite{kubota}, \cite{kuroda}, cf. \cite{ben-san-sny})\label{kubota}
Let $L/\mathbb{Q}$ be a totally real bi-quadratic extension, with unit group $E(L)$. Let $L_1, L_2$ and $L_3$ be the quadratic subfields of $L$. Let $\eps_i$ be the fundamental unit of $L_i$, for $i=1,2$ and $3$. Let $Q(L) := [E(L) : \langle -1, \eps_{1}, \eps_{2}, \eps_{3} \rangle]$ be the Hasse unit index of $L$. Then we have
\begin{equation}
\#A(L) = \dfrac{1}{4}\cdot Q(L)\cdot \#A(L_1)\cdot \#A(L_2)\cdot \#A(L_3).
\end{equation}
Further, the following are the possible systems of fundamental units of $L$ under some numbering of the fields $L_i$.
\begin{multicols}{2}
\begin{enumerate}
    \item $\{\eps_1, \eps_2, \eps_3 \}$
    \item $\{\sqrt{\eps_1}, \eps_2, \eps_3 \}$
    \item $\{ \sqrt{\eps_1}, \sqrt{\eps_2}, \eps_3\}$
    \item $\{\sqrt{\eps_1\eps_2}, \eps_2, \eps_3 \}$
    \item $\{\sqrt{\eps_1\eps_2}, \eps_2, \sqrt{\eps_3} \}$
    \item $\{ \sqrt{\eps_1\eps_2}, \sqrt{\eps_1\eps_3}, \sqrt{\eps_2\eps_3}\}$
    \item  $\{ \sqrt{\eps_1\eps_2\eps_3}, \eps_2, \eps_3 \}$
\end{enumerate}
\end{multicols}
\noindent Here, any $\eps_i$ ($i =1,2,3$) that appears under the square-root is assumed to have norm equal to $1$, except for the $7$th case, where all of $\eps_i$ ($i =1,2,3$) must have the same norm, either all $1$, or all $-1$.
\end{theorem}

The existence of infinitely many real quadratic fields of the form (\ref{cond1}) and (\ref{cond2}) follows easily from Dirichlet's theorem on primes in arithmetic progression using the Chinese remainder theorem. In particular, the following proposition ensures validity of Greenberg's conjecture for infinitely many real quadratic fields arising out of  Corollary \ref{lambda of 5,3,3,-1}.

\begin{propn}(cf. Proposition 2.2, \cite{CLS_2 class group}) \label{infinite primes}
Let $t \geq 1$ be an integer. Assume that for each $i \in \{1,\ldots,t\}$, we are given integers $a_i \in \{1,3,5,7 \}$, and for each $1 \leq j < k \leq t$, the integers $\eps_{kj} \in \{ \pm 1 \}$ are specified. Then there exist infinitely many $t$-tuples $\{ p_1,\ldots,p_t \}$ of prime numbers such that $p_i \equiv a_i \pmod {8}$ and the Legendre symbol $\left( \dfrac{p_k}{p_j} \right)$ equals $\eps_{kj}$.
\end{propn}

\section{The $2$-class group of $K=\Q(\sqrt{p_1q_1q_2})$}


For any number field $K$, the group $A(K)$ is isomorphic to the Galois group of $L(K)/K$, where $L(K)$ is the $2$-Hilbert class field of $K$. One of the important subfields of $L(K)$ containing $K$ is the genus field $K_G$ of $K$. It is defined as the maximal abelian extension of $\Q$ contained in $L(K)$. Moreover, the Galois group of $K_G/K$ is isomorphic to the $2$-torsion of $A(K)$ and the $2$-rank of Gal($K_G/K$) is equal to the $2$-rank of $A(K)$. The genus field $K_G$ can be obtained from the narrow genus field $K_G^{+}$. The narrow genus field $K_G^{+}$ is the maximal abelian extension of $\Q$ contained in the narrow $2$-Hilbert class field of $K$. If $K$ is an imaginary quadratic field, we have $K_G = K_G^{+}$. If $K$ is a real quadrtaic field, $K_G$ is the maximal real subfield of $K_G^{+}$. The narrow genus field of a quadratic number field $K$ can be explicitly determined from the prime factorization of the discriminant $D_{K}$. We can express the prime factorization of $D_{K}$ by $D_{K} = \pm 2^{e}p_1^{\ast}\cdots p_t^{\ast}$, where $e = 0,2$ or $3$, and
\begin{equation*}
  p_i^{\ast} = 
      \begin{cases}
        p_i & \text{if} \  p_i \equiv 1 \pmod 4\\
        -p_i & \text{if}\ p_i \equiv 3 \pmod 4.
      \end{cases}     
\end{equation*}
In this notation, we have $K_G^{+} = \mathbb{Q}( \sqrt{d}, \sqrt{p_1^{\ast}}, \cdots, \sqrt{p_t^{\ast}})$. 

We now prove a lemma concerning the order of $A(K)$ for $K = \Q(\sqrt{p_1q_2q_2})$, where the primes $p_1$, $q_1$ and $q_3$ satisfy $p_1 \equiv 1 \Mod{4}$ and  $q_1, \;q_2 \equiv 3 \Mod{4}$.

\begin{lemma}\label{A_0 = 2}
Let $K = \Q(\sqrt{p_1q_1q_2})$ such that $p_1 \equiv 1 \Mod{4}$ and $q_1,\; q_2 \equiv 3 \Mod{4}$. Then, $\#A(K) = 2$ if and only if $-1 \in \left\lbrace \left(\dfrac{q_1}{p_1}\right), \left(\dfrac{q_2}{p_1}\right) \right\rbrace $.
\end{lemma}
\begin{proof}
Consider the field $K = \Q(\sqrt{p_1q_1q_2})$. From the congruence modulo 4 conditions on the prime factors of the discriminant of $K$, the narrow genus field $K_G^{+}$ of $K$ turns out to be $\Q( \sqrt{p_1q_1q_2}, \sqrt{p_1}, \sqrt{-q_1}, \sqrt{-q_2})$. As $K$ is real, the genus field $K_G$ of $K$ is equal to $\Q( \sqrt{p_1q_1q_2}, \sqrt{p_1}, \sqrt{q_1q_2}) = Q(\sqrt{p_1q_1q_2}, \sqrt{p_1})$. We note that $K_G$ is a quadratic extension of $K$. Since ${\rm{rank}}_{2}$Gal($K_G/K$) =  ${\rm{rank}}_{2}A(K)$, we deduce that ${\rm{rank}}_{2}A(K) = 1$, and hence, $A(K)$ is cyclic. 

\smallskip

We first prove the forward part of the result, assuming that $\left(\dfrac{q_1}{p_1}\right) = -1$. A similar proof holds true if the other Legendre symbol is (or both the symbols are) equal to $-1$. Since $\left( \dfrac{q_1}{p_1} \right) = -1$, $q_1$ is inert in the extension $\Q(\sqrt{p_1})/\Q$. It follows that the prime $\mathfrak{q}_1$ in $K$ which lies above $q_1$ is inert in $K_G/K$. Thus,  $\mathfrak{q}_1$ is not totally split in $L(K)/K$, where $L(K)$ is the 2-Hilbert class field of $K$. By class field theory, $\mathfrak{q}_1$ is a non-principal ideal in $K$ and thus, $[\mathfrak{q}_1]$ must be of order $2$. Let $\phi : A(K) \rightarrow \rm{Gal}$($K_G/K$) be the Artin map. Since $\mathfrak{q}_1$ does not split in $K_G$, the Artin symbol $\left(\dfrac{K_G/K}{\mathfrak{q}_1}\right)$ must be non-trivial. Therefore, $[\mathfrak{q}_1]$ does not belong to Ker($\phi$). Let $G$ be the group Gal($K/\Q$), and $\sigma$ be its generator. Since $h(\Q)=1$, $\sigma$ acts as $-1$ on $A(K)$, which produces the relation $A(K)^{\sigma -1} = A(K)^2$. Besides, $A(K)^G = A(K)[2] = \{ id , [ \mathfrak{q}_1] \}$ as $A(K)$ is cyclic and $\sigma$ acts as -1 on $A(K)$. By Artin map, $A(K)/A(K)^{\sigma -1}$ is isomorphic to Gal($K_G/K$). Therefore, $[ \mathfrak{q}_1] \not\in A(K)^{\sigma -1} = A(K)^2$. Thus, $A(K) = A(K)^2 \cup [\mathfrak{q}_1]A(K)^2 = \{ id , [ \mathfrak{q}_1] \}A(K)^2$. Hence, by Nakayama's lemma, $A(K) = \{ id , [ \mathfrak{q}_1] \}$, which is of order $2$.

\smallskip

Conversely, suppose $\#A(K) = 2$, but $ \left(\dfrac{q_1}{p_1}\right) =1$ and $\left(\dfrac{q_2}{p_1}\right) =1$. From the order of $A(K)$, we gather that $L(K) = K(G) = \Q(\sqrt{p_1q_1q_2}, \sqrt{p_1})$. If both the Legendre symbols are equal to 1, then the primes $\mathfrak{p}_1, \mathfrak{q}_1$ and $\mathfrak{q}_2$ which lie above $p_1, q_1$ and $q_2$ respectively in $K$ are totally split in the extension $K_G/K = L(K)/K$. Thus, all three prime ideals must be principal in $K$, and $[\mathfrak{p}_1] = [\mathfrak{q}_1] = [\mathfrak{q}_2]$ in $A(K)$. Therefore, there exists $\alpha \in K^{\times}$ such that $\mathfrak{q}_1 = \langle \alpha \rangle \mathfrak{q}_2$. Squaring both sides and using the fact that the generators of a principal ideal differ by a factor of a unit, we obtain that $q_1 = \alpha^2q_2\eps^{n}$, where $\eps$ is the fundamental unit of $K$ and $n \in \Z$. If $n$ is even, then this implies that $\sqrt{\frac{q_1}{q_2}} \in K$, which is a contradiction. Therefore, $n$ must be odd. In that case, $K(\sqrt{\eps}) = K(\sqrt{\frac{q_1}{q_2}}) = K(\sqrt{p_1})$. Since $ [\mathfrak{p}_1] = [\mathfrak{q}_1]$, following the same argument, we get $K(\sqrt{\eps}) = K(\sqrt{\frac{q_1}{p_1}}) = K(\sqrt{q_2}) \neq K(\sqrt{p_1})$. That way, we again arrive at a contradiction. Hence, at least one of $\left(\dfrac{q_1}{p_1}\right)$ and $\left(\dfrac{q_2}{p_1}\right)$ must be equal to -1. 
\end{proof}

\section{The $2$-class group of $\Q(\sqrt{2p_1q_1q_2})$}
For a real quadratic field $K =\Q(\sqrt{p_1q_1q_2})$ where the three prime factors satisfy conditions (\ref{cond1}) or (\ref{cond2}) of Section 1, we shall use $F$ to denote the field $\Q(\sqrt{2p_1q_1q_2})$. 
The orders of $A_0$ and $A(F)$ can help in estimating the order of $A_1$, as we shall see more generally in  Lemma \ref{bound of A_{n+1}}. In this section, we examine the structure of $A(F)$. 
The discriminant $D_F$ is equal to $8p_1q_1q_2$, and has two prime factors that are congruent to $3$ modulo $4$. Thus, the genus field $F_G$ of $F$ is equal to $\Q(\sqrt{2}, \sqrt{p_1},\sqrt{q_1q_2})$. We note that Gal($F_G/F$) is isomorphic to $\Z/2\Z \oplus \Z/2\Z$. Hence, the $2$-rank of $A(F)$ is equal to $2$. In order to compute the order of $A(F)$ under certain Legendre symbol criteria on the prime factors, we shall use a result of R\'{e}idei and Reichardt (cf. \cite{RR-theorem}). Their result allows us to calculate the $2$-rank and the $4$-rank of the narrow 2-class groups, which in turn will help us in realising the structure of $A(F)$.

The $4$-rank of a finite abelian group $G$ is the $2$-rank of the quotient group $2G/4G$. We say that a group is $2$-elementary if it is isomorphic to the external direct product of some finitely many copies of $\Z/2\Z$. Clearly, an abelian $2$-group $G$ is $2$-elementary if and only if $ \#\left(2G/4G\right) = 1$. Let $L$ be any quadratic number field and $A^{+}(L)$ be the $2$-Sylow subgroup of the narrow class group  of $L$. Let $D_L$ be the discriminant of $L$ and, $S_1(L)$ and $S_2(L)$ be the sets defined as follows:
\begin{align*}
S_1(L) &:= \{ (D_1, D_2): |D_1| < |D_2|,\ D_{L}= D_1D_2,\ D_i \equiv 0 \mbox{ or } 1\Mod{4} \}, \\
S_2(L) &:= \{(1, D_{L})\} \ \cup \{ (D_1, D_2) \in S_1(L): \chi_{ _{D_1}}(p) = 1\;\;\forall  p \mid D_2 \text{\ and } \chi_{_{D_2}}(p) = 1 \;\;\forall  p \mid D_1 \},
\end{align*}
where $\chi_{_{D_{i}}}(p) = \left(\dfrac{D_{i}}{p}\right)$ is the Kronecker symbol for $i = 1 \mbox{ and } 2$. The Kronecker symbol is defined as follows. Let $n \in \Z$, with prime factorisation $n = up_1^{a_1}\cdots p_r^{a_r}$, where $ u \in \{1, -1\}$, $a_i \geq 1$ and $p_i$'s are prime numbers for $i = 1, \ldots, r$. Then, for any $m \in \Z$,  $\left(\dfrac{m}{n}\right) = \left(\dfrac{m}{u}\right)\displaystyle\prod_{i=1}^{r}\left(\dfrac{m}{p_i}\right)^{a_i}$. Here, $\left(\dfrac{m}{p_i}\right)$ is the usual Legendre symbol if $p_i$ is odd, $\left(\dfrac{m}{2}\right) = 
\begin{cases}
    1 & \text{if } m \equiv \pm 1 \Mod{8},\\
    -1 & \text{if } m \equiv \pm 3 \Mod{8}
\end{cases}$, and $\left(\dfrac{m}{u}\right) = 
\begin{cases}
    1 & \text{if } u = 1,\\
    -1 & \text{if } u = -1.
\end{cases}$

\noindent We now state the result by R\'{e}dei and Reichardt.

\begin{theorem}(\cite{RR-theorem}, \cite[Theorem 2.4]{mizu_thesis})\label{RR}
With the entities defined as above, we have $$\# S_1(L) = \#\left(A^{+}(L) / 2A^{+}(L)\right) ~ \mbox{ and }  ~ \#S_2(L) = \# \left(2A^{+}(L) / 4A^{+}(L)\right).$$
\end{theorem}
\begin{rmk}
We observe from Theorem \ref{RR} that $A^{+}(L)$ is $2$-elementary if and only if $\#S_2(L) = 1$. In that case, we have $=\#A^{+}(L) = \#S_1(L)$.  
\end{rmk}  
For any number field $L$, the class group $\mathcal{C}l(L)$ and the narrow class group $\mathcal{C}l^{+}(L)$ are all abelian groups. Also, the $2$-Sylow subgroup $A(L)$ of $\mathcal{C}l({L})$ can be viewed as a quotient group of the $2$-Sylow subgroup $A^{+}(L)$ of $\mathcal{C}l^{+}({L})$. Therefore, if $A^{+}(L)$ is a $2$-elementary group, then so is $A(L)$. For the fields of our choice, even the converse is true by the following proposition.

\begin{propn}\cite[Proposition 2.1]{CLS_2 class group}\label{A(K) and A+(K) are 2 elmtry}
Let $L = \Q(\sqrt{d})$ be a quadratic field, where $d \geq 1$ is a square-free integer having a prime divisor which is congruent to $3 \pmod 4$. If $A(L)$ is $2$-elementary, then so is $A^{+}(L)$.
\end{propn}

We are now in a position to prove that $A(F)$ is $2$-elementary for quadratic fields $F = \Q(\sqrt{2p_1q_1q_2})$, where the primes $p_1$, $q_1$ and $q_2$ follow one of the congruence conditions (\ref{cond1}) or (\ref{cond2}) together with some constraints on Legendre symbols. 

\begin{lemma}\label{order A(F)(5,3,3)}
Let $F = \Q(\sqrt{2p_1q_1q_2})$ with $\ p_1 \equiv 5 \Mod{8},\ q_1 \equiv 3 \Mod{8}$, and $\ q_2 \equiv 3 \Mod{8}$. Then the 2-class group $A(F)$ is of the form $\Z/2\Z\oplus \Z/2\Z$ if and only if one of the following conditions hold:
\begin{enumerate}
\item $\left( \dfrac{q_1q_2}{p_1}\right) = -1$.
\item $\left( \dfrac{q_1}{p_1}\right) = 1$ and $\left( \dfrac{q_2}{p_1}\right) = 1$.
\end{enumerate}
\end{lemma}
\begin{proof}
By Proposition \ref{A(K) and A+(K) are 2 elmtry}, it suffices to prove the equivalence for $A^{+}(F)$ in place of $A(F)$. For $A^{+}(F)$, we can exploit Theorem \ref{RR}. The discriminant of $F$ is equal to $8p_1q_2q_2$, which can be expressed in the following ways as $D_1D_2$, where $D_i \equiv 0,1 \Mod{4}$ and $|D_1| < |D_2|$. 
\begin{equation*}
(1, 8pq_1q_2),\ (8, pq_1q_2),\ (p, 8q_1q_2),\ (-q_1, -8pq_2),\ (-q_2 ,-8pq_1),\ (8p, q_1q_2),\ (-8q_1, -pq_2),\ (-8q_2, -pq_1).
\end{equation*}
These tuples account for the elements of the set $S_1(F)$. We now enlist the Kronecker symbols corresponding to each tuple in $S_1(F)$ other than $(1, 8p_1q_1q_2)$ to see which of these belong to the set $S_2(F)$.
\begin{table}[hbt!]
\begin{center}
\caption{Kronecker symbols corresponding to each element in $S_1(F)$}
\smallskip
\label{table1 RR}
\begin{tabular}{|c| c| c| } 
\hline 
\rule{0pt}{2ex}Sr. No. & Tuple & Kronecker Symbols  \\
\hline
\rule{0pt}{4ex}$1$ & $(8, p_1q_1q_2)$  & $\left(\dfrac{2}{p_1}\right),\left(\dfrac{2}{q_1}\right),\left(\dfrac{2}{q_2}\right),\left(\dfrac{p_1q_1q_2}{2}\right)$ \\
\rule{0pt}{4ex}$2$ & $(p_1, 8q_1q_2)$ & $\left(\dfrac{p_1}{2}\right),\left(\dfrac{p_1}{q_1}\right),\left(\dfrac{p_1}{q_2}\right),\left(\dfrac{2q_1q_2}{p_1}\right)$ \\
\rule{0pt}{4ex}$3$ & $(-q_1, -8p_1q_2)$ & $\left(\dfrac{-q_1}{2}\right),\left(\dfrac{-q_1}{p_1}\right),\left(\dfrac{-q_1}{q_2}\right),\left(\dfrac{-2p_1q_2}{q_1}\right)$ \\
\rule{0pt}{4ex}$4$ & $(-q_2, -8p_1q_1)$ & $\left(\dfrac{-q_2}{2}\right),\left(\dfrac{-q_2}{p_1}\right),\left(\dfrac{-q_2}{q_1}\right),\left(\dfrac{-2p_1q_1}{q_2}\right)$ \\
\rule{0pt}{4ex}$5$ & $(8p_1, q_1q_2)$ & $ \left(\dfrac{2p_1}{q_1}\right),\left(\dfrac{2p_1}{q_2}\right),\left(\dfrac{q_1q_2}{2}\right),\left(\dfrac{q_1q_2}{p_1}\right)$ \\
\rule{0pt}{4ex}$6$ & $(-8q_1, -p_1q_2)$ & $ \left(\dfrac{-2q_1}{p_1}\right),\left(\dfrac{-2q_1}{q_2}\right),\left(\dfrac{-p_1q_2}{2}\right),\left(\dfrac{-p_1q_2}{q_1}\right)$ \\
\rule{0pt}{4ex}$7$ & $(-8q_2, -p_1q_1)$ & $ \left(\dfrac{-2q_2}{p_1}\right), \left(\dfrac{-2q_2}{q_1}\right),\left(\dfrac{-p_1q_1}{2}\right),\left(\dfrac{-p_1q_1}{q_2}\right)$ \\[2ex]
\hline
\end{tabular}
\end{center}
\end{table}\\
In each case of the aforementioned criteria on Legendre symbols, we notice that there is at least one symbol that has value $-1$ in each row of Table \ref{table1 RR}. This implies that in each of the cases, order of $S_2(F)$ is equal to $1$. This means that $A^{+}(F)$, and hence $A(F)$, are both $2$-elementary. As ${\rm{rank}}_{2}A(F) =2$, it indeed must be isomorphic to $\Z/2\Z \oplus \Z/2\Z$.

\smallskip

Conversely, assuming $A(F)$ is $2$-elementary, by Proposition \ref{A(K) and A+(K) are 2 elmtry}, $A^{+}(F)$ must be $2$-elementary. In that case, at least one entry from each row in Table \ref{table1 RR} should be equal to $-1$. By supposing that at least one entry is $-1$, we exactly obtain the options mentioned in this lemma.
\end{proof}

Following the same approach, we obtain the next result (irrespective of any Legendre symbol restrictions).

\begin{lemma}\label{order A(F)(5,7,3)}
Let $F = \Q(\sqrt{2p_1q_1q_2})$ with $\ p_1 \equiv 5 \Mod{8},\ q_1 \equiv 7 \Mod{8}$, and $\ q_2 \equiv 3 \Mod{8}$. Then the $2$-class group $A(F)$ is isomorphic to $\Z/2\Z\oplus \Z/2\Z$. 
\end{lemma}

\section{$2$-class groups of the sub-extensions of $K_\infty/K$}

While inspecting the field $K = \Q(\sqrt{pq})$, where $p \equiv 3 \Mod{8}$ and $q \equiv 9 \Mod{16}$ with some additional conditions, Kumakawa in \cite[Lemma 2.1]{kumakawa2} derived an upper bound on the order of $A(K_{n+1})$ for all $n \geq 0$ in terms of the orders of $2$-class groups of subfields of $K_{n+1}$. More precisely, the subfields involved were $K_n$ and $K_n^{\prime}$, where $K_n^{\prime}$ denotes the subfield of $K_{n+1}$ containing $\Q_n$, different from $K_n$ and $\Q_{n+1}$. For example, if $K =\Q(\sqrt{d})$, then $K_0^{\prime} = \Q( \sqrt{2d}),\ K_1^{\prime} = \Q\left(\sqrt{(2 + \sqrt{2})d}\right)$, and so on. In the spirit of Kumakawa's work, we extract a tighter upper bound for the fields satisfying conditions (\ref{cond1}) and (\ref{cond2}) of Section 1.
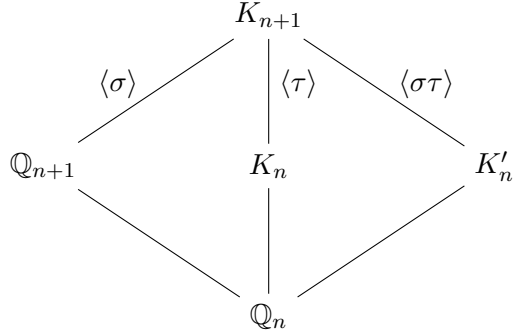
\begin{center}
\begin{figure}[hbt!] 
 \begin{tikzpicture}

    \node (Q1) at (0,0) {$\Q_n$};
    \node (Q2) at (3,2) {$K_n^{\prime}$};
    \node (Q3) at (0,2) {$K_n$};
    \node (Q4) at (-3,2) {$\Q_{n+1}$};  
    \node (Q5) at (0,4) {$K_{n+1}$};
    
    \draw (Q1)--(Q2);
    \draw (Q1)--(Q3); 
    \draw (Q1)--(Q4);
    \draw (Q2)--(Q5);
    \draw (Q3)--(Q5);
    \draw (Q4)--(Q5);
    
    \node (R1) at (2.1,3.1) {$\langle \sigma\tau \rangle$};
    \node (R2) at (0.4, 3.1) {$\langle \tau \rangle$};
    \node (R3) at (-2, 3.1) {$\langle \sigma \rangle$};
     \end{tikzpicture}
    \end{figure}
\end{center}

\begin{lemma}\label{bound of A_{n+1}}
Let $K = \Q(\sqrt{p_1q_1q_2})$ with $\ p_1 \equiv 5 \Mod{8}, \ q_1 \equiv 3 \Mod{4}$, and $\ q_2 \equiv 3 \Mod{8}$. Let $n \geq 0$. Suppose $\tau$ is the generator of \rm{Gal}$(K_{n+1}/K_n)$ and $\sigma$ is the generator of \rm{Gal}$(K_{n+1}/\Q_{n+1})$. Then $\#A_{n+1} \leq \#A_{n+1}^{\tau + 1}\cdot\#A(K_n^{\prime})/2.$ In particular, $\#A_{n+1} \leq \#A_n\cdot\#A(K_n^{\prime})/2$.
\end{lemma}
\begin{proof}
We shall use $[\mathfrak{a}]^{\sigma}$ to denote the the action of $\sigma$ on $[\mathfrak{a}]$. Since $K_{n+1}/\Q_n$ is a bi-quadratic extension, Gal($K_{n+1}/K_n^{\prime}$) = $ \langle \sigma\tau \rangle$. As $h(\Q_{n+1})$ is odd (cf. Theorem 10.4 of \cite{washington_book}), the norm map from $A_{n+1}$ to $\mathcal{C}l(\Q_{n+1})$ is trivial and hence, $\sigma$ acts as $-1$ on $A_{n+1}$. This implies that $A_{n+1}^{\sigma\tau -1} = A_{n+1}^{\tau +1}$, where $A_{n+1}^{\sigma\tau -1} = \{ [\mathfrak{a}]^{\sigma\tau}\cdot [\mathfrak{a}]^{-1}: [\mathfrak{a}] \in A_{n+1} \}$ and $A_{n+1}^{\tau +1}$ is defined similarly. We now consider the following exact sequence:
$$1 \longrightarrow A_{n+1}^{ \langle \sigma\tau \rangle} \longrightarrow A_{n+1} \longrightarrow A_{n+1}^{\sigma\tau -1} \longrightarrow 1.$$
Thus, we obtain $\#A_{n+1} = \#A_{n+1}^{ \langle \sigma\tau \rangle} \cdot \#A_{n+1}^{\sigma\tau -1} = \#A_{n+1}^{ \langle \sigma\tau \rangle}\cdot\#A_{n+1}^{\tau +1}.$ Now, applying the genus formula for the quadratic extension $K_{n+1}/K_n^{\prime}$, we have
$$\#A_{n+1}^{Gal( K_{n+1}/K_n^{\prime})}= \#A_{n+1}^{\langle \sigma\tau \rangle} = \dfrac{\#A(K_n^{\prime})\cdot 2^{t-1}}{[E(K_n^{\prime}): E(K_n^{\prime}) \cap N_{K_{n+1}/K_n^{\prime}}(K_{n+1}^{\times})]},$$
where $t$ is the number of primes ideals of $K_n^{\prime}$ ramified in $K_{n+1}$.  From the congruence modulo $8$ and $4$ conditions satisfied by $p_1, q_1$ and $q_2$, we note that $K/\Q$ is unramified at $2$. Note that $2$ is inert in $K/\Q$ when $q_1 \equiv 3\Mod{8}$, and splits when $q_1 \equiv 7\Mod{8}$. Since $2$ is unramified in $K/\Q$, the prime(s) above $2$ is (are) unramified in $K_n/\Q_n$ for all $n \geq 0$. Also, the prime(s) above $2$ is (are) ramified in the extension $K_n^{\prime}/\Q_n$ for all $n \geq 0$. Hence, the prime(s) above $2$ is (are) unramified in $K_{n+1}/K_n^{\prime}$. The primes above $p_1, q_1$ and $q_2$ are ramified in $K_{n+1}/\Q_{n+1}$ and $K_n^{\prime} /Q_n$, but not in the extension $\Q_{n+1}/\Q_n$. Combining all these, we conclude that $K_{n+1}/K_n^{\prime}$ is an unramified extension for all $n$ and hence, $t=0$. Therefore, $\#A_{n+1}^{Gal( K_{n+1}/K_n^{\prime})} \leq \#A(K_n^{\prime})/2$ and hence, $\#A_{n+1} \leq \#A_{n+1}^{\tau +1} \cdot \#A(K_n^{\prime})/2.$ Since $1+\tau$ acts as the norm map from $A_{n+1}$ to $A_n$, it follows that  $\#A_{n+1}^{\tau +1} \leq \#A_n$. Thus, $\#A_{n+1} \leq \#A_n\cdot\#A(K_n^{\prime})/2$. \end{proof}

\subsection*{Proof of Theorem \ref{(5,3,3,-1)}}
Suppose the primes $p_1, q_1$ and $q_2$ are congruent to $5$, $3$ and $3$ modulo $8$ respectively, along with $\left( \dfrac{q_1q_2}{p_1} \right)= -1$. We have $K_1 = \Q( \sqrt{2}, \sqrt{p_1q_1q_2})$, and from Remark \ref{A_n is cyclic}, $A_1$ is cyclic. By Lemma \ref{bound of A_{n+1}}, $\#A_1 \leq \#A_0 \cdot \#A(F)/2$, where $F=K_0^{\prime}=\Q(\sqrt{2p_1q_1q_2})$. Combining lemmas \ref{A_0 = 2} and \ref{order A(F)(5,3,3)}, we conclude that  $A_1$ is a cyclic group of order $2$ or $4$. Further, we note from Lemma \ref{bound of A_{n+1}} that the order of $A_1$ also depends on $A_1^{\tau +1}$, where Gal($K_1/K$) = $\langle \tau \rangle$. Since $A_1^{\tau +1} \subseteq A_1^{\langle \tau \rangle}$, we have $\# A_1^{\tau +1} \leq \# A_1^{\langle \tau \rangle}$. By the genus formula, $\# A_1^{\langle \tau \rangle} \leq \#A_0\cdot 2^{t-1}$, where $t$ is the number of places of $K$ ramified in $K_1$. From the congruence modulo 8 conditions, $D_K \equiv 5 \Mod{8}$, where $D_K$ is the discriminant of $K$. Consequently, the rational prime 2 is inert in $K/\Q$, and only one place of $K$ gets ramified in $K_1$. Therefore, $t =1$ and $\# A_1^{\tau +1}\leq \# A_1^{\langle \tau \rangle} \leq 2$. 

\smallskip 

If $\#A_1^{\langle \tau \rangle} = 1$, then $\#A_1 \leq 1 \cdot 4/2 = 2$ by Lemma \ref{bound of A_{n+1}}. With the $2$-rank of $A_1$ being $1$, order of $A_1$ must be $2$. Hence, $\#A_1 = \#A_0 = 2$.

\smallskip

Now suppose that $\#A_1^{\langle \tau \rangle} = 2$. We claim that $A_1$ cannot have order 4. Suppose on the contrary, $A_1 = \langle [\mathfrak{a}] \rangle$ such that $[\mathfrak{a}]$ has order $4$. In that case, $A_1^{\langle \tau \rangle} = \{ id, [\mathfrak{a}]^2 \}$. Since $A_1$ is a Gal($K_1/K$)-module, 
$[\mathfrak{a}]^{\tau}$ is either equal to $[\mathfrak{a}]$ or $[\mathfrak{a}]^{-1}$. If $[\mathfrak{a}]^{\tau} = [\mathfrak{a}]$, then $\#A_1^{\langle \tau \rangle} = 4,$ which is not true. Therefore, $[\mathfrak{a}]^{\tau} = [\mathfrak{a}]^{-1}$, and consequently, $A_1^{\tau +1} = \{ id \}$. We have $\#A_{1} \leq \#A_{1}^{\tau + 1}\cdot\#A(F)/2$ by Lemma \ref{bound of A_{n+1}}. It follows that $\# A_1 \leq  1\cdot 4/2 = 2$, which contradicts our assumption. Therefore, $\#A_1 = \#A_0 = 2$. $\hfill\Box$

\subsection*{Proof of Corollary \ref{lambda of 5,3,3,-1}}
Since the discriminant $D_K$ is congruent to $5$ modulo $8$, the prime $2$ is inert in $K/\Q$. Moreover, $2$ is ramified in $\Q_1/\Q$. Thus, the prime above $2$ is totally ramified in $K_1/K$. The same argument holds for any extension $K_{n}/K$ for all $n \geq 1$. Applying Theorem \ref{(5,3,3,-1)} and Theorem \ref{fukuda's result} together, $\#A_n = \#A_0 = 2$ for all $n \geq 0$. Thus, $A_n$ is isomorphic to $\Z/2\Z$ for all $n \geq 0$, and the Iwasawa module $X_{\infty}$ corresponding to the $\Z_2$-extension of $K$ is isomorphic to $\Z/2\Z$. It follows that the Iwasawa invariant $\lambda$ vanishes. 

\smallskip

When we look at the $\Z_2$-extension of $F$, we recognize that the fields at layers $n \geq 1$ are the same as the ones in the $\Z_2$-extension of $K$. As the order of the class group at each layer is $2$, the Iwasawa module associated with $F$ is also isomorphic to $\Z/2\Z$ and the corresponding $\lambda$-invariant vanishes. $\hfill\Box$

\subsection*{Proof of Theorem \ref{(5,3,3,1)}}

Let $\p_1$, $\q_1$ and $\q_2$ be the prime ideals above the rational primes $p_1$, $q_1$ and $q_2$ respectively in $K/\Q$, and $\p^{\prime}_1$, $\q^{\prime}_1$ and $\q^{\prime}_2$ be the corresponding ideals in $F/\Q$.  We employ Kuroda-Kubota's class number formula to get the desired result. In order to appeal to the formula, we need to evaluate the Hasse unit index $Q(K_1)$ which involves the fundamental units of $K$, $F$, and $\Q_1$ along with their square-roots. Let $\eps_1$, $\eps_2$ and $\eps_3$ be the fundamental units of $K$, $F$, and $\Q_1$ respectively. Suppose $\eps_1 = \frac{a + b\sqrt{p_1q_1q_2}}{2}$, where $a$ and $b$ are integers of same parity. If $N_{K/\Q}(\eps_1) = -1$, then taking modulo $q_1$ of the norm equation, we obtain $a^2 \equiv -4 \Mod{q_1}$, which suggests that $-1$ is a quadratic residue modulo $q_1$. This is not possible as $q_1 \equiv 3 \Mod{8}$. Therefore, $N_{K/\Q}(\eps_1) = 1 $, and likewise, $ N_{F/\Q}(\eps_2) = 1$. The fundamental unit $\eps_3 = 1 +\sqrt{2}$ has norm -1 over $\Q$. From all these norm values, we conclude from Theorem \ref{kubota} that the fundamental system of units of $K_1$ must be one of $\{ \eps_1, \eps_2, \eps_3 \}, \{ \sqrt{\eps_1}, \eps_2, \eps_3 \}, \{ \sqrt{\eps_1}, \sqrt{\eps_2}, \eps_3\}$, and $\{\sqrt{\eps_1\eps_2}, \eps_2, \eps_3 \}$. We now eliminate certain possibilities. For convenience, we provide our argument in two parts.

\smallskip

 {\bf Part 1.} Since $A(F)$ is 2-elementary (from Lemma \ref{order A(F)(5,3,3)}), its 2-Hilbert class field $L(F)$ and its genus field $F_G$ must be the same, which is the field $\Q(\sqrt{2}, \sqrt{p_1}, \sqrt{q_1q_2})$. The field $F_G$ has three subfields that are bi-quadratic over $\Q$ which contain $F$. These are, $L_1 (= K_1) := \Q(\sqrt{2}, \sqrt{p_1q_1q_2}), L_2 := \Q(\sqrt{p_1}, \sqrt{2q_1q_2})$, and $L_3 := \Q(\sqrt{2p_1}, \sqrt{q_1q_2})$. Let $\ell^{\prime}$ be the prime above 2 in $F$. Then from the congruence modulo 8 and Legendre symbol criteria, we observe that the prime $\p^{\prime}_1$ and $\ell^{\prime}$ split completely only in the extension $L_3/F$. Similarly, the primes $\q^{\prime}_1$ and $\q^{\prime}_2$ split completely only in $L_2/F$. Thus, the primes $\p^{\prime}_1$ and $\ell^{\prime}$ have the same decomposition field $L_3$ in the extension $L(F)/F$, and  
the primes $\q^{\prime}_1$ and $\q^{\prime}_2$ have the decomposition field $L_2$ in $L(F)/F$. Thus, by Artin map, $\left( \dfrac{L(F)/F}{\p^{\prime}_1}\right) = \left( \dfrac{L(F)/F}{\ell^{\prime}}\right)$. Since the map is taken with respect to the extension $L(F)/F$, the symbols being equal implies that $[\p^{\prime}_1] = [\ell^{\prime}]$. Thus, the two ideals differ by a principal fractional ideal, say $\langle \beta \rangle$, where $\beta \in F^{\times}$. Therefore, $\p^{\prime}_1 = \langle \beta \rangle \ell^{\prime}$, which upon squaring implies $\langle p_1 \rangle  = \langle 2\beta^2 \rangle$. Hence, there exists $n \in \Z$ such that $p_1 = 2\beta^2\eps_2^n$. If $n$ is even, then $\sqrt{p_1} \in K_1$, which is not possible. Therefore, $n$ must be odd. This produces the equality $\sqrt{\eps_2} = \beta_1\sqrt{p_1/2}$, where $\beta_1^{-1} = \beta\eps_2^{\frac{n-1}{2}} \in F$. Now, if $\sqrt{\eps_2} \in K_1$, then again, $\sqrt{p_1} \in K_1$, which is a contradiction. Hence, $\sqrt{\eps_2} \not\in K_1$. Also, we stress that $K_1(\sqrt{\eps_2}) = K_1(\sqrt{p_1})$.

Now we shall proceed to prove that $\sqrt{\eps_1}$ and $\sqrt{\eps_1\eps_2}$ do not belong to $K_1$ if the ideal $\p_1$ is not principal in $K$, and only $\sqrt{\eps_1\eps_2}$ belongs to $K_1$ if $\p_1$ is principal in $K$.

\smallskip

 {\bf Part 2.} Since $\left( \dfrac{q_1}{p_1} \right) = 1$ and $\left( \dfrac{q_2}{p_1} \right) = 1$, $\#A_0 \geq 4$ by Lemma \ref{A_0 = 2}. As each of the primes $p_1$, $q_1$ and $q_2$ are ramified in the extension $K/\Q$, the order of the ideal classes $[\p_1], [\q_1]$ and $[\q_2]$ must be at the most 2. Given that $A_0$ is cyclic, it has exactly one element of order 2. Thus, at least two ideal classes out of $[\p_1], [\q_1]$ and $[\q_2]$ must be equal. We shall achieve our goal of proving that square-roots of certain  fundamental units are not present in $K_1$ by making the following claims:

\smallskip

\noindent {\bf Claim 1}: If $\p_1$ is principal, then $\#A_1 = 2\cdot\#A_0$. \\
If $\p_1$ is principal, then it must be equivalent to the ideal $\langle 2 \rangle$ in $K$. Therefore, there exists $\alpha \in K^{\times}$ such that $\p_1  = \langle 2 \alpha \rangle$. As argued in Part 1, there exists $\alpha_1 \in K^{\times}$ such that $\sqrt{p_1} = 2\alpha_1\sqrt{\eps_1}$. This again implies that $\sqrt{\eps_1} \not\in K_1$, and $K_1(\sqrt{\eps_1}) = K_1(\sqrt{p_1}) = K_1(\sqrt{\eps_2})$. It follows that $\sqrt{\eps_1\eps_2}\in K_1$ as it is fixed under the action of the Galois group. 
Hence, from Part 1, the system of fundamental units of $K_1$ is $\{ \sqrt{\eps_1\eps_2}, \eps_2,\eps_3\}$. Thus, $Q(K_1) = 2$, and by Theorem \ref{kubota}, $\#A_1 = \frac{1}{4} \#A_0 \cdot \#A(F)\cdot\#A(\Q(\sqrt{2}))\cdot Q(K_1) = 2\cdot\#A_0$.

\noindent From Lemma \ref{bound of A_{n+1}}, it is evident that $\#A_1 \neq 2\cdot\#A_0$ is equivalent to $\#A_1 = \#A_0$. Thus, we register here that $\#A_1 = \#A_0$ implies $\p_1$ is not principal in $K$. 

\smallskip

\noindent {\bf Claim 2}: The ideals $\q_1$ and $\q_2$ cannot be simultaneously principal.\\
Suppose on the contrary, both the ideals are principal. Then each of the ideals must be equivalent to the ideal $\langle 2 \rangle$ in $K$. Proceeding as Part 1, we obtain $K(\sqrt{\eps_1}) = K(\sqrt{q_1}) = K(\sqrt{q_2})$, which is a contradiction as $K(\sqrt{q_1}) \neq K(\sqrt{q_2})$. Therefore, our claim stands true. In addition, if $[\q_1] = [\q_2]$, then both the classes must be of order 2.

\smallskip

\noindent {\bf Claim 3}: The ideal $\p_1$ is principal if and only if $[\q_1] = [\q_2]$ in $\mathcal{C}l(K)$.\\
Suppose $\p_1$ is principal. Then from Claim 1, $K_1(\sqrt{\eps_1}) = K_1(\sqrt{p_1})$. If $[\q_1] \neq [\q_2]$, then exactly one of $\q_1$ and $\q_2$ must be principal. Without loss of generality, suppose $\q_1$ is principal (similar arguments are applicable for $\q_2$). Then the ideals $\p_1$ and $\q_1$ must be equivalent as ideals and must differ by a factor of a principal fractional ideal. This yields that $K_1(\sqrt{\eps_1}) = K(\sqrt{\frac{p_1}{q_1}}) = K_1(\sqrt{q_2}) \neq K_1(\sqrt{p_1})$. This is a contradiction, and hence both the classes $[\q_1]$ and $[\q_2]$ have to be equal (which internally implies that the classes should be of order 2 because of Claim 2). 

\smallskip

\noindent Conversely, suppose $[\q_1] = [\q_2]$, and $\p_1$ is not principal. Then all the three classes $[\p_1], [\q_1]$, and $[\q_2]$ must be equal with order 2 (from Claim 2). Now following the previous technique, $[\p_1] = [\q_1]$ implies that $K_1(\sqrt{\eps_1} ) = K_1(\sqrt{q_2})$, and $[\p_1] = [\q_2]$ implies $K_1(\sqrt{\eps_1}) = K_1(\sqrt{q_1})$, which cannot occur in unison. Hence there is an inconsistency, which implies our claim.

\smallskip

\noindent {\bf Claim 4}: If $\p_1$ is not principal, then $\#A_1 = \#A_0$.

\noindent If $\p_1$ is not principal, then by Claim 3, $[\q_1] \neq [\q_2]$. Thus,  exactly one of $\q_1$ or $\q_2$ is principal. Without loss of generality, suppose $\q_1$ is that non-principal ideal. Then, $\p_1$ and $\q_1$ must be equivalent and therefore, following the lines of argument in Part 1, $\sqrt{\eps_1} \not\in K_1$ and $K_1(\sqrt{\eps_1}) = K_1(\sqrt{\frac{p_1}{q_1}}) = K_1(\sqrt{ q_2}) \neq K_1(\sqrt{\eps_2})$. For that reason, both $\sqrt{\eps_1}$ and $\sqrt{\eps_2}$ are not in $K_1$.

\smallskip

\noindent If $\sqrt{\eps_1\eps_2} \in K_1$, then $\sqrt{\eps_1\eps_2} \in K_1(\sqrt{\eps_1})$, and this means that $\sqrt{\eps_2} \in K_1(\sqrt{\eps_1})$. Similarly, $\sqrt{\eps_1} \in K_1(\sqrt{\eps_2})$. This leads to the equality $K_1(\sqrt{\eps_1}) = K_1(\sqrt{\eps_2})$. But this is absurd because $K_1(\sqrt{\eps_1})= K_1(\sqrt{ q_2}) \neq K_1(\sqrt{p_1}) = K_1(\sqrt{\eps_2})$. Therefore, $\sqrt{\eps_1\eps_2} \not\in K_1$. The fundamental system of units of $K_1$ is the set $\{ \eps_1, \eps_2, \eps_3 \}$, and the Hasse unit index $Q(K_1)$ is equal to 1. From Theorem \ref{kubota}, $\#A_1 = 1/4 \cdot \#A_0 \cdot \#A(F)\cdot\#A(\Q(\sqrt{2})\cdot Q(K_1) = 2 = \#A_0$. Thus, Claim 4 follows.

\smallskip

\noindent  We have $\p_1$ is principal implying that $\#A_1 \neq \#A_0$ (from Part 1 and  Claim 1), and $\p_1$ is not principal implying that $\#A_1 = \#A_0$ (from Part 1 and Claim 4). This completes the proof that $\p_1$ is principal if and only if $\#A_1 \neq \#A_0$.    $\hfill\Box$

\subsection*{Proof of Corollary \ref{lambda of (5,3,3,1)}}
If the ideal $\p_1$ is principal, then there exist integers $a$ and $b$ of same parity such that $N_{K/\Q}( \frac{a+b\sqrt{p_1q_1q_2}}{2}) = p_1$ or $-p_1$, i.e, $a^2 - b^2 p_1q_1q_2 = 4p_1$ or $-4p_1$. If the norm is equal to $-p_1$, then taking equation modulo $q_1$, we obtain $a^2 \equiv -4p_1 \Mod{q_1}$, which indicates that $-p_1$ is a quadratic residue modulo $q_1$. But this is impossible as $\left(\dfrac{q_1}{p_1}\right) = \left(\dfrac{p_1}{q_1}\right) = 1$, and $q_1 \equiv 3 \Mod{8}$. As a result, if the prime $\p_1$ is principal, then there must exist integers $a$ and $b$ of same parity such that $a^2 -b^2p_1q_1q_2 = 4p_1$. If there are no such integers, then $\p_1$ is not principal.

\smallskip

From the Legendre symbol values and Lemma \ref{A_0 = 2}, the group $A_0$ is cyclic with order at least $4$. If $\p_1$ is not principal in $K$, then from Theorem \ref{(5,3,3,1)}, $\#A_1 = \#A_0$. From the congruence modulo 8 conditions, the prime above 2 is totally ramified in $K_n/K$ for all $n \geq 1$. Thus, $A_n$ is isomorphic to $\Z/2^{m}\Z$ for some $m \geq 2$ and for all $n \geq 0$. Hence, we deduce that the Iwasawa module $X_{\infty}$ is isomorphic to $\Z/2^{m}\Z$ for some $m \geq 2$, and the Iwasawa invariant $\lambda_2$ is equal to 0. The same holds when we study $F$ instead of $K$. The associated Iwasawa module has the same structure, with vanishing $\lambda$-invariant.   $\hfill\Box$

\smallskip

We exhibit Theorem \ref{(5,3,3,1)} through some examples in Tables 2 and 3. The computations have been carried out through \texttt{SageMath}.\\
\begin{table}[hbt!]
\begin{center}
\caption{Fields $K = \Q(\sqrt{p_1q_1q_2})$ where the prime $\p_1$ above $p_1$ is not principal}
\smallskip
\label{tab:table2}
\begin{tabular}{|c|c|c|c|c|c|} 
\hline 
$p_1$  & $q_1$  & $q_2$  & $\# A_0$  & $\# A_1$ \\
\hline
$5$ & $11$ & $19$ &  $4$ & $4$ \\
\hline
$5$ & $11$ & $139$  & $4$ & $4$ \\
\hline
$5$ & $11$ & $179$ &  $4$ & $4$ \\
\hline
$5$ & $19$ & $211$  &  $4$ & $4$ \\
\hline
$13$ & $43$ & $107$ &  $4$ & $4$ \\
\hline
$13$ & $131$ & $107$ &  $8$ & $8$ \\
\hline
$13$ & $107$ & $131$ &  $4$ & $4$ \\
\hline
$29$ & $59$ & $107$ &  $8$ & $8$ \\
\hline
$29$ & $59$ & $67$ &  $4$ & $4$ \\
\hline
$29$ & $67$ & $83$ & $4$ & $4$ \\
\hline

\end{tabular}
\end{center}
\end{table}

\begin{table}[hbt!]
\begin{center}
\caption{Fields $K = \Q(\sqrt{p_1q_1q_2})$ where the prime $\p_1$ above $p_1$ is principal}
\smallskip
\label{tab:table3}
\begin{tabular}{|c|c|c|c|c|c|} 
\hline 
$p_1$  & $q_1$  & $q_2$  &  $\# A_0$  & $\# A_1$ \\
\hline
$5$ & $11$ & $131$ &  $4$ & $8$ \\
\hline
$5$ & $19$ & $59$  &  $4$ & $8$ \\
\hline
$5$ & $11$ & $211$ &  $4$ & $8$ \\
\hline
$5$ & $19$ & $139$  &  $4$ & $8$ \\
\hline
$5$ & $19$ & $179$ &  $4$ & $8$ \\
\hline
$13$ & $43$ & $179$ &  $8$ & $16$ \\
\hline
$29$ & $59$ & $83$ &  $4$ & $8$ \\
\hline
$29$ & $83$ & $107$ &  $4$ & $8$ \\
\hline
$29$ & $59$ & $227$ & $4$ & $8$ \\
\hline
$53$ & $11$ & $43$ &  $16$ & $32$ \\
\hline

\end{tabular}
\end{center}
\end{table}

\subsection*{Proof of Theorem \ref{(5,7,3)}}
The proof of Theorem \ref{(5,7,3)} predominantly follows the approach used in the proof of Theorem \ref{(5,3,3,1)}. We furnish the proof for the case  $\left( \dfrac{p_1}{q_2} \right) = 1$ as the proof for the other case is similar. When $\left( \dfrac{q_1q_2}{p_1} \right) = -1$, the prime $p_1$ is inert in $\Q( \sqrt{q_1q_2})$. Thus, the prime $\p_1$ above $p_1$ in $K$ is inert in the extension $K_G = K(\sqrt{p_1})$. Also, $\q_1$ is inert in $K_G/K$ as $\left( \dfrac{q_1}{p_1} \right) = -1$. Since $A_0$ is cyclic of order 2, $[\p_1] = [\q_1]$. Thus, $\sqrt{p_1} = \sqrt{q_1}\alpha_1\sqrt{\eps_1}$, for some $\alpha_1 \in K^{\times}$. This implicates that $\sqrt{\eps_1} \not\in K_1$, and $K_1(\sqrt{\eps_1}) = K_1(\sqrt{q_2})$.

\smallskip

We recall the fields $L_1$, $L_2$, and $L_3$ such that $F \subset L_i \subset F_G$ for $i =1,2,3$, defined in Part 1 of the proof of Theorem \ref{(5,3,3,1)}. The primes $\p_1^{\prime}$ and $\q_2^{\prime}$ above $p_1$ and $q_2$ in $F$ have the same decomposition field $L_2$. Thus, $[\p_1] = [\q_2]$ as their corresponding Artin symbols are equal with respect to the extension $F_G/F = L(F)/F$. Hence, we deduce that $\sqrt{\eps_2} \not\in K_1$, and $K_1(\sqrt{\eps_2}) = K_1(\sqrt{q_1})$. As explained in the last part of the proof of Theorem \ref{(5,3,3,1)}, since $K_1(\sqrt{\eps_1}) \neq K_1(\sqrt{\eps_2})$, $\sqrt{\eps_1\eps_2} \not\in K_1$. Again by Theorem \ref{kubota}, $Q(K_1) = 1$, and $\#A_1 = \#A_0 = 2$. By Fukuda's result on the stability of order of $A_n$ (Theorem \ref{fukuda's result}), $\#A_n = 2$ for all $n \geq 0$. Thus, $X_{\infty}$ is isomorphic to $\Z/2\Z$. The same occurs when we consider the field $F$ instead of $K$ as both these fields have the same $\Z_2$-extension, barring the base fields. Consequently, the Iwasawa module $X_{\infty}$ corresponding to $F$ is isomorphic to $\Z/2\Z$ and its $\lambda$-invariant vanishes. $\hfill\Box$

\section{Concluding Remarks}
In this section, we shall focus on the case $p_1 \equiv 5 \Mod{8}$, $q_1 \equiv 7 \Mod{8}$, $q_2 \equiv 3 \Mod{8}$, and $\left( \dfrac{q_1}{p_1} \right) = 1$. We find that the prime $\q_1^{\prime}$ splits completely in all the fields $L_i$, $i=1,2,3$. Thus, the ideal $\q_1^{\prime}$ is principal in $F$  by class field theory. Accordingly, there exists some $\alpha \in F^{\times}$ such that $\q_1^{\prime} = \langle 2\alpha \rangle$. Again squaring both sides and equating the resultant principal ideals, we obtain $\sqrt{q_1} = 2\alpha_1\sqrt{\eps_2}$ for some $\alpha_1 \in F^{\times}$. That being so, $\sqrt{\eps_2} \not\in K_1$, and $K_1(\sqrt{\eps_2}) = K_1(\sqrt{q_1})$. 

\smallskip

{\bf Case 1:} If $\left(\dfrac{q_2}{p_1}\right) = -1$, then the ideals $\p_1$ and $\q_2$ are non-principal in $K$, and $A_0$ is cyclic of order 2 by Lemma \ref{A_0 = 2}. Thus, $[\p_1] = [\q_2]$, which leads to $\sqrt{\eps_1} \not\in K_1$, and $K_1(\sqrt{\eps_1}) = K_1(\sqrt{q_1}) = K_1(\sqrt{\eps_2})$. Therefore, $\sqrt{\eps_1\eps_2} \in K_1$, and $Q(K_1) = 2$, which brings about the relation $\#A_1 = 2\cdot \#A_0$. Since it has been proven in \cite{mouhib-mova} that the Iwasawa invariant $\lambda_2$ of $K$ is equal to 0, $X_{\infty}$ is finite and cyclic. Additionally, $\#A_1 = 2\cdot\#A_0 = 4$ implies that $X_{\infty}$ is isomorphic to $\Z/2^m\Z$, for some $m \geq 2$.

\smallskip

{\bf Case 2:} If $\left(\dfrac{q_2}{p_1}\right) = 1$, then we have two subcases, depending on whether $\mathfrak{q}_1$ is principal or not.  If $\q_1$ is principal in $K$, then from the equivalence of the principal ideals $\q_1$ and $\langle 2 \rangle$, we can prove that $\sqrt{\eps_1}$ and $\sqrt{\eps_2}$ do not belong to $K_1$, $\sqrt{\eps_1\eps_2} \in K_1$, $Q(K_1) = 2$, and $\#A_1 = 2\cdot \#A_0$. From Lemma \ref{A_0 = 2}, $\#A_0 \geq 4$, and thus, $\#A_1 \geq 8$, and finally, $X_{\infty}$ is of the form $\Z/2^{m}\Z$, for some $m \geq 3$. \\
As in the proof of Theorem \ref{(5,3,3,1)}, we make the following claims to prove that $\#A_1 = \#A_0$ if and only if $\q_1$ is not principal in $K$:\\
{\bf Claim 1.} If $\q_1$ is principal, then $\#A_1 = 2 \cdot\#A_0$. Therefore, if $\#A_1 = \#A_0$, then $\mathfrak{q}_1$ is non principal in $K$.\\
{\bf Claim 2.} The ideals $\p_1$ and $\q_2$ cannot be simultaneously principal in $K$.\\
{\bf Claim 3.} The ideal $\q_1$ is principal in $K$ if and only if $[\p_1] = [\q_2]$ in $\mathcal{C}l(K)$.\\
{\bf Claim 4.} If the ideal $\q_1$ is not principal in $K$, then $\#A_1 = \#A_0$. Thus, $A_n$ is isomorphic to $A_0$ for all $n \geq 1$. \\
These claims merge to prove that $X_{\infty}$ is isomorphic to $\Z/2^{m}\Z$ for some $m\geq 2$ when $\left( \dfrac{q_1}{p_1} \right) = \left( \dfrac{q_2}{p_1} \right) = 1$, and the ideal $\q_1$ is not principal in $K$. Therefore from both the cases, we observe that when $\left( \dfrac{q_1}{p_1} \right) = 1$, $X_{\infty}$ is not just finite and cyclic, but also, its order must be greater than or equal to 4.

{\bf Acknowledgements.} The authors would like to sincerely thank the anonymous referee, whose valuable comments have helped in improving the manuscript. The authors also take immense pleasure to thank Indian Institute of Technology Guwahati for providing excellent facilities to carry out this research. The research of the second author is partially funded by the MATRICS, SERB research grant MTR/2020/000467.


\begin{thebibliography}{9999}

\bibitem{azizi}
A. Azizi and A. Mouhib, {\it Sur le Rang du 2-Groupe de Classes de ou un Premier}, {\sf Transactions of the American Mathematical Society}, {\bf 353} (2001), 2741-2752.



\bibitem{ben-lem-sny}
E. Benjamin, F. Lemmermeyer and C. Snyder, {\it On the unit group of some multiquadratic number fields}, {\sf Pacific Journal of Mathematics}, {\bf 230} (2007), 27-40.

\bibitem{ben-san-sny}
E. Benjamin, F. Sanborn and C. Snyder, {\it Capitulation in unramified quadratic extensions of real quadratic number fields}, {\sf Glasgow Mathematical Journal}, {\bf 36} (1994), 385-392.



\bibitem{brown-parry}
E. Brown and C J. Parry, {\it The 2-class group of certain biquadratic number fields}, {\sf J. Reine Angew. Math.}, {\bf 295} (1977), 61-71.

\bibitem{asjc-rama}
J. Chattopadhyay and A. Saikia, {\it Simultaneous indivisibility of class numbers of pairs of real quadratic fields}, {\sf Ramanujan J.}, {\bf 58} (2022), 905-911.

\bibitem{CLS_2 class group}
J. Chattopadhyay, H Laxmi, and A. Saikia, {\it On the structure and stability of ranks of $2$-class groups in cyclotomic $\Z_2$-extensions of certain real quadratic fields},  {\sf Res. number theory}, {\bf 9} (2023), https://doi.org/10.1007/s40993-023-00478-2.

\bibitem{chevalley}
Chevalley C, {\it Sur la th\'{e}orie du corps de classes dans les corps finis et les corps locaux (Th\'{e}se)}, {\sf J. Faculty of Sciences Tokyo}, {\bf 2} (1933), 365-476.



\bibitem{ferrero-washington}
B. Ferrero and L C. Washington, {\it The Iwasawa invariant $\mu_{p}$ vanishes for abelian number fields}, {\sf Annals of Mathematics}, {\bf 109} (1979), 377-395.

\bibitem{fukuda}
T. Fukuda, { \it Remarks on $\Z_p $-extensions of number fields}, {\sf Proceedings of the Japan Academy, Series A, Mathematical Sciences}, {\bf 70} (1994), 264-266.

\bibitem{fukuda-komatsu}
T. Fukuda and K. Komatsu, {\it On the Iwasawa $\lambda$-Invariant of the Cyclotomic $\Z_2$-Extension of a Real Quadratic Field}, {\sf Tokyo Journal of Mathematics}, {\bf 28} (2005), 259-264.


\bibitem{gras2}
G. Gras, {\it On the $\lambda$ stability of $p$-class groups along cyclic $p$-towers of a number field}, {\sf International Journal of Number Theory}, {\bf 18} (2022), 2241-2263. 


\bibitem{greenberg}
R. Greenberg, {\it  On the Iwasawa invariants of totally real number fields}, {\sf Amer. J. Math.}, {\bf 98} (1976), 263-284.


\bibitem{ichimura2}
H. Ichimura and H. Sumida, {\it On the Iwasawa invariants of certain real abelian fields}, {\sf Tohoku Mathematical Journal, Second Series} {\bf 49} (1997), 203-215.

\bibitem{iwasawa}
K. Iwasawa, {\it On $\Gamma$-extensions of algebraic number fields}, {\sf Bulletin of the American
Mathematical Society}, {\bf 65} (1959), 183-226.


\bibitem{kraft-schoof}
J. Kraft and R. Schoof, {\it Computing Iwasawa modules of real quadratic number fields}, {\sf Compositio Math.}, {\bf 97} (1995), 135-155.

\bibitem{kubota}
T. Kubota, {\it \"{U}ber den bizyklischen biquadratischen Zahlk\"{o}rper}, {\sf Nagoya Mathematical Journal}, {\bf 10} (1956), 65-85.


\bibitem{kumakawa2}
N. Kumakawa, {\it On the Iwasawa $\lambda$-invariant of the cyclotomic $\Z_2$-extension of $\Q(\sqrt{pq})$ and the $2$-part of the class number of $\Q(\sqrt{pq}, \sqrt{2 + \sqrt{2}})$}, {\sf International Journal of Number Theory}, {\bf 17} (2021), 931-958.

\bibitem{kuroda}
S. Kuroda, {\it \"{U}ber den Dirichletschen K\"{o}rper}, {\sf J. Fac. Sci. Imp. Univ. Tokyo Sec. I.}, {\bf 4} (1943), 383-406.

\bibitem{mizu_thesis}
Y. Mizusawa, {\it A Study of Iwasawa Theory on Class Field Towers}, {\sf PhD Thesis (Waseda University)}, (2004). 

\bibitem{mizu_paper}
Y. Mizusawa, {\it On the Iwasawa invariants of $\Z_2$-extensions of certain real quadratic fields}, {\sf Tokyo Journal of Mathematics}, {\bf 27} (2004), 255-261.

\bibitem{mizusawa2}
Y. Mizusawa, {\it On unramified Galois 2-groups over $\Z_2$-extensions of real quadratic fields}, {\sf Proceedings of the American Mathematical Society}, {\bf  138} (2010), 3095-3103.

\bibitem{mouhib}
A. Mouhib, {\it The structure of the unramified abelian Iwasawa module of some number fields}, {\sf Pacific Journal of Mathematics}, {\bf 323} (2023), 173-184.

\bibitem{mouhib-mova}
A. Mouhib and A. Movaheddi, {\it Cyclicity of the unramified Iwasawa module}, {\sf Manuscripta Mathematica}, {\bf 135} (2011), 91-106.


\bibitem{nishino}
Y. Nishino, {\it On the Iwasawa Invariants of the Cyclotomic $\Z_2$-Extensions of Certain Real Quadratic Fields}, {\sf Tokyo Journal of Mathematics}, {\bf 29} (2006), 239-245.

\bibitem{ouyang-zhang}
Y. Ouyang and Z. Zhang, {\it Hilbert genus fields of real biquadratic fields}, {\sf The Ramanujan Journal}, {\bf 37} (2015), 345-363.

\bibitem{ozaki-taya}
M. Ozaki and H. Taya, {\it On the Iwasawa $\lambda_2$-invariants of certain families of real quadratic fields}, {\sf Manuscripta Math.}, {\bf 94} (1997), 437-444.

\bibitem{RR-theorem}
L. R\'{e}dei  and H. Reichardt, {\it Die durch vier teilbaren Invarianten der Klassengruppe der quadratischen Zahlk{\"o}rper}, {\sf J. reine angew. Math.}, {\bf 170} (1933), 59-74.


\bibitem{washington_book}
L C. Washington, {\it Introduction to cyclotomic fields},  {\sf Springer Science \& Business Media}, {\bf 83} (1997).

\bibitem{yamamoto}
G. Yamamoto, {\it On the vanishing of Iwasawa invariants of absolutely abelian p-extensions}, {\sf Acta Arithmetica}, {\bf 94} (2000), 365-371.

\end{thebibliography}
\end{document}